\documentclass[a4paper,11pt]{amsart}

\usepackage[
textwidth=15cm,
textheight=24cm,
hmarginratio=1:1,
vmarginratio=1:1]{geometry}

\usepackage{amssymb}

\newcommand{\D}{\mathbb D}
\newcommand{\C}{\mathbb C}
\newcommand{\T}{\mathbb T}
\newcommand{\BMOA}{\mathit{BMOA}}
\newcommand{\VMOA}{\mathit{VMOA}}

\newtheorem{theorem}{Theorem}
\newtheorem{lemma}[theorem]{Lemma}
\newtheorem{prop}[theorem]{Proposition}
\newtheorem{corollary}[theorem]{Corollary}
\newtheorem*{claim}{Claim}

\theoremstyle{remark}
\newtheorem{remark}[theorem]{Remark}

\numberwithin{equation}{section}
\numberwithin{theorem}{section}

\title{Structural rigidity of generalised Volterra operators on $H^p$}

\author{Santeri Miihkinen}
\address{Miihkinen: Department of Mathematics, \r{A}bo Akademi University, FI-20500 \r{A}bo,
Finland}
\email{santeri.miihkinen@abo.fi}
\author{Pekka J.\ Nieminen}
\address{Nieminen: Department of Mathematics and Statistics,
FI-20014 University of Turku, Finland}
\email{pekka.nieminen@utu.fi}
\author{Eero Saksman}
\address{Saksman: Department of Mathematics and Statistics, Box 68,
FI-00014 University of Helsinki, Finland}
\email{eero.saksman@helsinki.fi}
\author{Hans-Olav Tylli}
\address{Tylli: Department of Mathematics and Statistics, Box 68,
FI-00014 University of Helsinki, Finland}
\email{hans-olav.tylli@helsinki.fi}


\date{October 8, 2017}

\begin{document}

\begin{abstract}
We show that the non-compact generalised analytic Volterra operators $T_g$, where $g \in \BMOA$,
have the following structural rigidity property on the Hardy spaces $H^p$ for $1 \le p < \infty$
and  $p \neq 2$:\ if $T_g$ is bounded below on an infinite-dimensional subspace $M \subset H^p$,
then $M$ contains a subspace linearly isomorphic to $\ell^p$. This implies in particular that
any Volterra operator $T_g\colon H^p \to H^p$ is $\ell^2$-singular for $p \neq 2$.
\end{abstract}

\maketitle

\section{Introduction}

In this paper we establish some structural rigidity properties
of the generalised  Volterra-type integral operators
\[
 f \mapsto   T_g f(z) = \int_0^z f(w)g'(w)\,dw,
   \qquad z \in \D,
\]
 on the Hardy spaces $H^p$ for $p \neq 2$, where  $g$ is a given analytic function defined
 on the open unit disc $\D$ of the complex plane.
Various properties of this class of operators have been investigated
since the mid-1990s on $H^p$, as well as on many other function spaces;
see e.g.\ the surveys \cite{Aleman06} and \cite{Siskakis04}. In particular,
$T_g$  is bounded (respectively, compact) $H^p \to H^p$ if and only
if $g \in \BMOA$ (respectively, $g \in \VMOA$) according to results of Aleman and Siskakis
\cite{AleSis} for $1 \le p < \infty$ and Aleman and Cima \cite{AleCi} for $0 < p < 1$.  

Our aim is to study largeness in the linear qualitative sense of the Volterra operators
$T_g\colon H^p \to H^p$ for $p \neq 2$. 
We first recall the following general concept for
a Banach space $X$.
 A bounded operator $U\colon X \to X$ is said to
fix a copy of the given Banach space $E$, if
there is a closed subspace $M \subset X$, linearly isomorphic to $E$, and $c > 0$ so that
$\Vert Ux\Vert \ge c \cdot \Vert x\Vert$ for all $x \in M$ (that is, the restriction
$U_{|M}$ defines an isomorphism  $M\to UM$). In this direction, the first author \cite{Miihkinen}
showed that any non-compact operator $T_g$ fixes a copy of $\ell^p$ in $H^p$.
Our main result demonstrates that the non-compactness of
these operators is, in fact, very  limited in the qualitative sense.

\begin{theorem} \label{thm:main0}
Let $g \in \BMOA$ be arbitrary and $1 \le p < \infty$, $p \neq 2$. If the restriction
${T_g}_{|M}$ is bounded below on an infinite-dimensional subspace $M \subset H^p$, then
$M$ contains a linearly  isomorphic copy of $\ell^p$.
\end{theorem}

We refer to Remark \ref{comm}.(1) below for a more detailed discussion of the contents 
of Theorem \ref{thm:main0}. For the moment recall only 
that for $1 \le q < \infty$ the bounded operator $U\colon X \to X$
is called \textit{$\ell^q$-singular},
denoted  $U \in \mathcal{S}_q(X)$, if $U$ does not fix any copies of $\ell^q$ in $X$.
The following corollary is an immediate consequence
of Theorem~\ref{thm:main0} and the fact that the sequence spaces
$\ell^p$ and $\ell^2$ are totally incomparable
for $p \neq 2$, see e.g.\ \cite[2.a.1 and p.~54]{LT77}.
It also answers a question from \cite[Sec.~3]{Miihkinen} in the negative.
Note further that Theorem~\ref{thm:main0} is obvious
for $H^2$, while the corollary fails for $p = 2$.

\begin{corollary} \label{cor1}
Let $1 \leq p < \infty$ and $p \neq 2$. Then
$T_g \in  \mathcal{S}_2(H^p)$
for every $g \in \BMOA$, that is, $T_g$ does not fix any linearly isomorphic
copies of $\ell^2$ in $H^p$.
\end{corollary}

In particular, $T_g\colon H^p\to H^p$ cannot fix a copy of the whole space $H^p$ itself
in the case $p\neq 2$.

It is of interest to contrast the preceding results with the behaviour
of the composition operators $f \mapsto C_\phi f = f \circ \phi$
where $\phi\colon \D \to \D$ is any fixed analytic map. It is well known that
$C_\phi$ is a  bounded operator $H^p \to H^p$, cf.\ \cite[Thm~3.1]{CMC}.
It was recently shown \cite{LNST} that for $p \neq 2$ any non-compact 
$C_\phi\colon H^p \to H^p$  fixes an isomorphic copy of $\ell^p$
in $H^p$ (akin to the case of the Volterra operators). However, there are many examples
of composition operators $C_\phi$ that fix a copy of $\ell^2$ in $H^p$ for $p \neq 2$ and
this class can be characterised in terms of the boundary behaviour of the symbol $\phi$,
see \cite[Thm~1.4]{LNST}.

Theorem~\ref{thm:main0} will be proved in Section~\ref{proof}.
In Section~\ref{Ces} we sketch as an appetiser an
argument for the $\ell^2$-singularity of the Ces{\`a}ro operator on $H^p$, or equivalently, the 
$\ell^2$-singularity of the non-compact Volterra operator $T_h$ obtained with the symbol
$h(z) = \log \frac{1}{1-z} \notin \VMOA$.
We consider it worthwhile to look at this classical example separately, since it
exemplifies a ``localisation'' of the non-compactness, which somewhat surprisingly turns out
to hold  for arbitrary  symbols $g$ in $\BMOA$ in a suitable formulation.

We refer to \cite{Duren}, \cite{Garnett} and \cite{Girela}  for
unexplained notions and results for the spaces $H^p$ and $\BMOA$,
and respectively \cite{LT77} and \cite{AK} for general Banach space theory.
We recall here only that the analytic map 
$f\colon \D \to \C$ belongs to $\BMOA$ if 
$\vert f\vert_{*} = \sup_{a \in \D} \Vert f \circ \sigma_a - f(a)\Vert_{H^2} < \infty$,
where $\sigma_a(z) = \frac{a-z}{1-\overline{a}z}$ for $z \in \D$, 
and that $\BMOA$ is a Banach space equipped with the norm 
$\Vert f \Vert = \vert f(0) \vert + \vert f\vert_{*}$. Moreover, $\VMOA$ is the closed subspace 
of $\BMOA$ consisting of the maps $f \in \BMOA$ for which 
$\lim_{\vert a\vert \to 1} \Vert f \circ \sigma_a - f(a)\Vert_{H^2} = 0$.

\section{Appetiser:\ $\ell^2-$singularity of the Ces{\`a}ro operator}
\label{Ces}

As a prototypical example of a non-compact Volterra-type integral operator we briefly consider
the classical Ces{\`a}ro operator
\[
Cf(z) = \frac{1}{z} \int_0^z \frac{f(w)}{1-w}\,dw = \frac{1}{z}T_h f(z), \quad z \in \D,
\]
where $h(z) = \log \frac{1}{1-z}$. Thus $T_h = M_zC$,
where the multiplier $f \mapsto M_zf = zf$ is an isometry on $H^p$.
A systematic study of $C$ on $H^p$ was initiated by Siskakis; see e.g.\ \cite{Siskakis87}
and \cite{Siskakis90}.

Let $\T = \partial\D$ be the unit circle.
Observe that  $h'(z) = \frac{1}{1-z}$ is bounded on $\T\setminus J$, where $J$ is any open
arc containing the point $z = 1$, and its values peak at the point $z = 1$. Intuitively speaking,
$C$ behaves like a compact operator on a major part of $\T$ and in a non-compact
fashion  near $z = 1$. It is known by \cite[Thm~1.1]{Miihkinen} that $C$ fixes a copy of
$\ell^p$.  We next outline a proof of the following special case of Corollary~\ref{cor1}.

\begin{claim}
The Ces{\`a}ro operator $C$ is $\ell^2$-singular on $H^p$ for $1 < p < \infty$ and $p \neq 2$.
\end{claim}

\begin{proof}
Assume to the contrary that $C$ fixes a copy of $\ell^2$. Thus there is a sequence $(f_n)$
in $H^p$ which is equivalent to the
unit vector basis of $\ell^2$, and such that $C$ is bounded below on
the closed linear span of $(f_n)$.
Put $E_\varepsilon = \{e^{it}: |e^{it}-1| < \varepsilon\}$ for
$\varepsilon > 0$. By the absolute continuity of the measures
$E \mapsto \int_E |Cf_n|^pdm$, we get that
\begin{equation}\label{Ces1}
  \lim_{\varepsilon \to 0}\int_{E_\varepsilon}|Cf_n|^p\,dm = 0
\end{equation}
for each $n$. We next claim that
\begin{equation}\label{Ces2}
\lim_{n \to \infty}\int_{\T\setminus E_\varepsilon}|Cf_n|^p \,dm = 0
\end{equation}
for each $\varepsilon > 0$. In fact, given $0 < \delta < 1$, we have
\begin{equation}\label{Ces3}
  |Cf_n(e^{i\theta})| \le \int_0^{1-\delta}
  \frac{|f_n(re^{i\theta})|}{|1-re^{i\theta}|}\,dr +
  \int_{1-\delta}^1 \frac{|f_n(re^{i\theta})|}{|1-re^{i\theta}|}\,dr,
\end{equation}
where the first  integral in \eqref{Ces3} tends to $0$ as $n \to \infty$,
since $(f_n)$ converges to zero uniformly on compact subsets of $\D$ (recall that
$f_n \to 0$ weakly in $H^p$). Moreover,  for $e^{i\theta} \notin E_\varepsilon$
the second integral in \eqref{Ces3} can be made uniformly as small as desired by
choosing $\delta > 0$ small enough, and using the estimates
$|1-re^{i\theta}| \ge c\varepsilon$ for $\varepsilon \le |\theta| \le \pi$
and some $c > 0$, as well as
\begin{equation}\label{eval}
  |f_n(z)| \le \frac{2^{1/p}\|f_n\|_{H^p}}{(1-|z|)^{1/p}}, \qquad z \in \D,
\end{equation}
for $p\in (0,\infty)$, see \cite[p.~36]{Duren}. Hence \eqref{Ces2} holds.

Since conditions \eqref{Ces1} and \eqref{Ces2} are satisfied,
and $\|Cf_n\|_{H^p} \approx 1$, we may proceed as in the proof of Proposition~3.5
in \cite{Miihkinen} (cf.\ also the proof of Theorem 1.2 in \cite{LNST}) and extract a
subsequence $(Cf_{n_k})$ which is equivalent to the unit vector basis of $\ell^p$ in $H^p$.
However, since $(Cf_{n_k})$ is also equivalent to the unit vector basis of
$\ell^2$ by assumption, the above contradicts the fact that $\ell^p$ and $\ell^2$
are totally incomparable for $p \neq 2$. This contradiction shows that
$C$ is indeed $\ell^2$-singular.
\end{proof}

As noted before,  the above concrete operator $T_h \in  \mathcal{S}_2(H^p)$ 
acts like a compact operator on most of $\T$.
We now face the problem of establishing a similar behaviour for  $T_g$ 
with an arbitrary  symbol $g \in \BMOA \setminus \VMOA$. 
In the general case sheer size estimates of the derivative $g'$ will not suffice.
For instance, by \cite{Rud55} there is a bounded analytic function
$g \in H^{\infty}$, such that the curves $\{g(re^{it}): 0 \le r < 1\}$ are unrectifiable
for almost every $e^{it} \in \T$. Instead, to overcome the basic difficulty
we will provide 
a key result for  the localisation of  non-compactness, see Proposition \ref{prop:KTg} below. 
Our starting point  is a condensation phenomenon (Lemma~\ref{le:littleCarleson}) 
shared by the Carleson measures $\vert g'(z)\vert^2 \log\frac{1}{\vert z\vert}\,dA(z)$, which allows us to isolate the ``support of the non-compactness'' on $\T$.
Nonetheless, exploiting this knowledge is not straightforward, and thus we also need  to post-compose $T_g$ in Lemma~\ref{le:CphiTg} by an auxiliary composition operator, whose symbol is geometrically carefully chosen  with respect to the ``support of the non-compactness.''

\section{Proof of Theorem~\ref{thm:main0}}
\label{proof}

We will actually establish the following more precise version of Theorem~\ref{thm:main0}.

\begin{theorem} \label{thm:main}
Let $g \in \BMOA$ be arbitrary and $1 \le  p < \infty$. Suppose that $(f_n)$ is a sequence 
in $H^p$ such that $\|f_n\|_{H^p} = 1$ for all $n$ and $f_n \to 0$ uniformly on the compact
subsets of $\D$. If
\begin{equation}\label{con1}
  \liminf_{n\to\infty} \|T_g f_n\|_{H^p} > 0,
\end{equation}
then there is a subsequence $(f_{n_j})$ such that $(T_g f_{n_j})$ is equivalent to the unit
vector basis of $\ell^p$, that is, there is constant  $c > 0$ for which
\begin{equation}\label{ellp}
   c^{-1} \cdot \Bigl( \sum_j |\alpha_j|^p \Bigr)^{1/p}
   \leq \Bigl\| \sum_j \alpha_j T_g f_{n_j} \Bigr\|_{H^p}
   \leq c \cdot \Bigl( \sum_j |\alpha_j|^p \Bigr)^{1/p}
\end{equation}
holds for all $(\alpha_j) \in \ell^p$.
\end{theorem}

Assuming Theorem~\ref{thm:main} for a moment we next derive Theorem~\ref{thm:main0}
from it.

\begin{proof}[Proof of Theorem~\ref{thm:main0}]
Suppose that the restriction ${T_g}_{|M}$ is bounded below on a closed
infinite-dimensional subspace
$M \subset H^p$. We may next pick a normalised sequence
$(f_n) \subset M$ such that $f_n \to 0$ uniformly as $n\to\infty$ on the compact subsets
of $\D$.  For the reader's convenience we recall that this is obtained by choosing  a normalised sequence $(f_n) \subset M$
such that  $f_n^{(r)}(0) = 0$ for all $r = 0, \ldots, n$ and any $n \geq 1$.
The choice is possible since $M$ is infinite-dimensional and the intersection of the kernels of the
evaluation functionals $f \mapsto f^{(r)}(0)$ for $r = 0,\ldots,n$ has
finite codimension in $M$. It is then a standard fact that
$f_n \to 0$ as $n \to \infty$ on the compact subsets of $\D$.
In fact, since $f \mapsto z^nf$ is an isometry on $H^p$, we deduce by \eqref{eval}
that for any given $0 < r < 1$ there is a constant $C(r)$ so that for any $z$ satisfying
$\vert z\vert \le r$ one has
\[
  \vert f_n(z)\vert \le C(r) \vert z\vert^n  \le C(r)  r^n \to 0
\]
as $n \to \infty$.

After these preparations Theorem \ref{thm:main} produces a subsequence $(f_{n_j})$
for which \eqref{ellp} holds for $(T_g f_{n_j})$. Since $T_g$ is bounded below on
the closed linear span $[\{f_{n_j}: j \geq 1\}] \subset M$ of
$(f_{n_j})$, it follows that $T_g$ fixes a linearly isomorphic copy of $\ell^p$ in $M$.

\end{proof}

The argument for Theorem \ref{thm:main} is based on the following proposition,
which shows that the non-compact behaviour  of $T_g$ for the symbols
$g \in \BMOA \setminus \VMOA$ is ultimately concentrated on
subsets of arbitrarily small measure of the unit circle $\T$.
As usual we view $H^p$ as a closed subspace of $L^p(\T) = L^p(\T,m)$,
where $m$ is the normalised Lebesgue measure on $\T$.

\begin{prop} \label{prop:KTg}
Let $g \in \BMOA$ and $1 \le p < \infty$. Then for every $\eta > 0$ there
exists a compact set $E \subset \T$ such that $m(\T \setminus E) < \eta$
and $\chi_E T_g\colon H^p \to L^p(\T)$ is a compact operator.

In particular, for any bounded sequence $(f_n) \subset H^p$ such that
$f_n \to 0$ uniformly on compact subsets of $\D$ one has
$\| (\chi_E T_g)f_n \|_{L^p} \to 0$ as $n \to \infty$.
\end{prop}

Taking Proposition~\ref{prop:KTg} for granted momentarily, we next show how it yields
Theorem~\ref{thm:main}.

\begin{proof}[Proof of Theorem~\ref{thm:main}]
By a repeated application of Proposition~\ref{prop:KTg} we can
find compact sets $E_1 \subset E_2 \subset \cdots \subset \T$
such that $m(\T\setminus E_k) \to 0$ as $k \to \infty$ and
$\chi_{E_k}T_g\colon H^p \to L^p(\T)$ is a compact operator for each $k$.
Moreover, since $f_n \to 0$ uniformly  on the compact subsets of $\D$ by
assumption, Proposition~\ref{prop:KTg} ensures that
\begin{equation}\label{con2}
   \| (\chi_{E_k}T_g)f_n \|_{L^p}^p
   = \int_{E_k} |T_g f_n|^p\,dm \to 0
   \quad\text{as $n \to \infty$}
\end{equation}
for each $k$. On the other hand, for each $n$ we have by absolute continuity that
\begin{equation}\label{con3}
   \int_{\T\setminus E_k} |T_g f_n|^p \,dm \to 0
   \quad\text{as $k \to \infty$.}
\end{equation}

By assumption \eqref{con1} we may suppose that
$\int_{\T} |T_g f_n|^p \,dm > d^p$ for each $n$
and a suitable constant $d > 0$.
For any fixed $\delta > 0$, we may by successive applications of \eqref{con2} and \eqref{con3}
extract increasing indices $n_1 < n_2 < \cdots$ and $j_1 < j_2 < \cdots$ such that
the conditions
\begin{equation}\label{tg1}
  \Bigl(\int_{\T \setminus E_{j_{r}}} |T_g f_{n_{s}}|^p \,dm \Bigr)^{1/p} < 4^{-r}\delta d,
  \qquad s = 1, \ldots, r-1,
\end{equation}
\begin{equation}\label{tg2}
  \Bigl(\int_{E_{j_{r}}} |T_g f_{n_{r}}|^p \,dm \Bigr)^{1/p} < 4^{-r}\delta d
\end{equation}
\begin{equation}\label{tg3}
  \Bigl(\int_{\T \setminus E_{j_{r}}} |T_g f_{n_{r}}|^p \,dm \Bigr)^{1/p} > d
\end{equation}
hold for all $r \geq 1$. This is a straightforward gliding hump type argument: Suppose  that we have found  indices
$n_1 < \ldots < n_r$ and $j_1 < \ldots < j_r$ so that
\eqref{tg1} - \eqref{tg3} hold until $r$. Next we use
\eqref{con3} to get $j_{r+1} > j_r$
so that \eqref{tg1} holds for $f_{n_{1}}, \ldots , f_{n_{r}}$ and the set $\T \setminus E_{j_{r+1}}$,
and then \eqref{con2} to find $n_{r+1} > n_r$ so that \eqref{tg2} holds.
Moreover, we may ensure \eqref{tg3} at stage $r+1$ since
$\int_{\T} |T_g f_n|^p \,dm > d^p$ for each $n$.

Finally, by applying the perturbation argument
from  \cite[Sec.~3]{Miihkinen} (see also Theorem 1.2 in \cite{LNST}),
which we do not repeat here, it follows that \eqref{ellp} holds for $(T_g f_{n_j})$
once $\delta > 0$ is small enough.
\end{proof}

We next turn to the proof of the crucial Proposition~\ref{prop:KTg}.
To any analytic function
$g\colon \D \to \C$ we associate the positive Borel measure $\mu_g$ on
$\D$ defined by the density
\[
   d\mu_g(z) = |g'(z)|^2 \log \frac{1}{|z|} \,dA(z),
\]
where $A$ is the Lebesgue area measure normalised such that $A(\D) = 1$.
Then the Littlewood-Paley identity (see e.g.\ \cite[Thm~2.30]{CMC}) implies
that
\[
  \lVert g \rVert_{H^2}^2 = |g(0)|^2 + 2 \int_\D d\mu_g(z)
\]
for $g \in H^2$. We also recall that $g \in \BMOA$ if and only if $\mu_g$ is a
Carleson measure, i.e.\ there is a constant $c > 0$ such that
\[
   \mu_g(W(\zeta,h)) \leq ch
   \quad\text{for $\zeta \in \T,\ 0<h<1$,}
\]
where $W(\zeta,h)$ is the Carleson window
\[
   W(\zeta,h) = \{ z\in\D : 1-h<|z|<1,\: |\arg (z/\zeta) |<h \}.
\]
Furthermore, $g \in \VMOA$ if and only if $\mu_g$ is a vanishing
Carleson measure, i.e.
\[
   \sup_{\zeta\in\T} \mu_g(W(\zeta,h)) = o(h)
   \quad\text{as $h\to 0$.}
\]
For proofs of the above results see e.g \cite[Chap.~VI.3]{Garnett} or
\cite[Sec.~6]{Girela}.

Our first auxiliary result says that every function $g \in \BMOA$
(or even $g \in H^2$) has uniformly vanishing mean oscillation on $\T$ up to
a set of arbitrarily small measure.

\begin{lemma} \label{le:littleCarleson}
Let $g \in H^2$ be arbitrary. Then for every $\varepsilon > 0$, there exists a
compact set $K \subset \T$ such that $m(\T \setminus K) < \varepsilon$ and
\[
   \sup_{\zeta\in K} \mu_g(W(\zeta,h)) = o(h)
   \quad\text{as $h\to 0$.}
\]
\end{lemma}

\begin{proof}
For each $k \geq 1$, let $\nu_k$ be the projection to the unit circle of the
measure $\mu_g$ restricted to the annulus $S_k = \{z: 1-\frac{1}{k} < |z| < 1\}$.
That is, $\nu_k$ is determined by the condition
\[
  \nu_k (I(\zeta,h)) = \mu_g (\{ z\in S_k : |\arg(z/\zeta)|<h \})
\]
for all boundary arcs $I(\zeta,h) = \{\xi\in\T : |\arg(\xi/\zeta)|<h\}$.
Consider the Hardy-Littlewood maximal function of $\nu_k$,
\[
   \nu_k^*(\zeta) = \sup_{0<h<\pi} \frac{1}{2h} \nu_k(I(\zeta,h)).
\]
By the maximal function theorem it satisfies
\begin{equation} \label{eq:maximal}
   m(\{ \zeta\in\T : \nu_k^*(\zeta) > \lambda\}) \leq C\frac{\nu_k(\T)}{\lambda}
\end{equation}
for all $\lambda > 0$ and a numerical constant $C > 0$. Note that here
$\nu_k(\T) = \mu_g(S_k) \to 0$ as $k\to\infty$ since $\mu_g$ is a finite
measure by the Littlewood-Paley identity.

We claim that $\nu_k^* \to 0$ almost everywhere on $\T$ as $k\to\infty$.
In fact, since
$(\nu_k^*)$ is obviously a pointwise decreasing sequence, we would otherwise
find a constant $\lambda > 0$ and a set $E \subset \T$ of positive measure
such that $\nu_k^*(\zeta) > \lambda$ for all $\zeta \in E$ and $k \geq 1$.
But this would contradict \eqref{eq:maximal}. Hence $\nu_k^* \to 0$ a.e.\ on $\T$.
Egorov's theorem now implies that there is a set $F \subset \T$ with
$m(\T\setminus F) < \varepsilon/2$ such that $\nu_k^* \to 0$
uniformly in $F$ as $k \to \infty$. Since, for every $k \geq 1$ and
$\zeta \in F$,
\[
   \sup_{0<h<1/k} \frac{\mu_g(W(\zeta,h))}{h}
   \leq \sup_{0<h<\pi} \frac{\nu_k(I(\zeta,h))}{h} = 2\nu_k^*(\zeta),
\]
we deduce that $\sup_{\zeta\in F} \mu_g(W(\zeta,h)) = o(h)$ as
$h \to 0$. Finally, it just remains to pick a compact subset
$K\subset F$ with $m(F\setminus K) < \varepsilon/2$.
\end{proof}
To exploit Lemma~\ref{le:littleCarleson} in the analysis of the
Volterra operator $T_g$ we will employ the product operator $C_\phi T_g$
for a composition operator $C_\phi$ whose symbol $\phi$ will be associated to the
compact set $K$ from Lemma~\ref{le:littleCarleson}. Here
\[
  (C_\phi T_g)f(z) =  \int_0^{\phi(z)} f(w)g'(w)\,dw,
  \qquad z \in \D,
\]
for any analytic function $f\colon \D \to \C$. Recall that both $C_\phi$ and $T_g$, and hence the product $C_\phi T_g$, are bounded on $H^p$ for all $p\in (0,\infty).$ For each $\zeta \in \T$
define the Stolz domain $S(\zeta)$ in $\D$ with vertex at $\zeta$ as the
interior of the convex hull of the set
$\{z: |z| < \tfrac{1}{2}\} \cup\{\zeta\}$.

\begin{lemma} \label{le:CphiTg}
Let $g \in \BMOA$ and $\varepsilon > 0$. Choose a compact subset $K \subset \T$ as in
Lemma~\ref{le:littleCarleson} and put $\Omega = \bigcup_{\zeta\in K} S(\zeta)$.
Let $\phi$ be a Riemann map from $\D$ onto $\Omega$ with $\phi(0)=0$.
Then $C_\phi T_g\colon H^p \to H^p$ is a compact operator for any $1 \le p < \infty$.
\end{lemma}

\begin{proof}
We start by considering the case $p = 2$.
By the Littlewood-Paley identity, the fact that $\phi(0) = 0$, and the change of
variables $w = \phi(z)$, we get that
\begin{equation}\label{LP}
   \|(C_\phi T_g) f\|_{H^2}^2 =
   2 \int_\Omega |f(w)|^2|g'(w)|^2 \log\frac{1}{|\phi^{-1}(w)|} \,dA(w).
\end{equation}
Since $|\phi^{-1}(w)| \geq |w|$ for all $w \in \Omega$ by Schwarz's lemma, it will be enough to
show that $\chi_\Omega\,d\mu_g$ is a vanishing Carleson measure.
Assuming this for a moment, known results (cf.\ the proof of \cite[Thm~2.33]{CMC}) yield
that the natural embedding
$H^2 \to L^2(\D, \chi_\Omega\,d\mu_g)$ is a compact operator, whose norm pointwise dominates that of
$C_\phi T_g\colon H^2 \to H^2$ by \eqref{LP}. This easily implies that also $C_\phi T_g\colon H^2 \to H^2$ is compact.

In order to verify the vanishing nature of the Carleson measure  $\chi_\Omega\,d\mu_g$, let   $\zeta \in \T$ and $0 < h < \tfrac{1}{4}$, and consider a Carleson
window $W(\zeta,h)$ which has a nonempty intersection with $\Omega$. Note that  $\partial \Omega\cap \T =K$, and let
$\xi$ be a point of $K$ that is closest to $\zeta$ (if $\zeta \in K$,
then take $\xi = \zeta$). It follows by geometric inspection that
the angular distance from $\zeta$ to $\xi$ is less than $2h$.
Therefore $W(\zeta,h) \subset W(\xi,3h)$ and so
\[
   \frac{1}{h} (\chi_\Omega\mu_g)(W(\zeta,h))
   \leq 3 \cdot \frac{1}{3h} \mu_g(W(\xi,3h)).
\]
In view of Lemma~\ref{le:littleCarleson} this implies that $\chi_\Omega\,d\mu_g$ is
a vanishing Carleson measure.

The claim for the other values of $p$ is obtained from the case $p=2$ above and the
identification $H^p = (H^{p_0}, H^{p_1})_{\theta, p}$ in terms of real interpolation spaces,
where $\frac{1}{p} = \frac{1-\theta}{p_0} + \frac{\theta}{p_1}$ with $0 < p_0 < \infty$ and
$p_1 = 2$ (see e.g.\ \cite[p.~1]{KX}), and one-sided
Krasnoselskii-type interpolation of compactness for operators. For the case $p > 1$ the
classical form \cite{K} of this result suffices, and for $p = 1$ we refer e.g.\ to 
\cite[Thm~3.1]{CP}
for a version of Krasnoselskii's theorem that also applies to quasi-Banach spaces
such as $H^{p_0}$ for $0 < p_0 < 1$.
\end{proof}

The auxiliary steps in Lemmas \ref{le:littleCarleson} and \ref{le:CphiTg} finally allow us
to establish  Proposition~\ref{prop:KTg}.

\begin{proof}[Proof of Proposition~\ref{prop:KTg}]
Fix $0 < \eta < 1/2$. Let $K$, $\Omega$ and $\phi$ be as given by Lemmas
\ref{le:littleCarleson} and \ref{le:CphiTg} corresponding to $\varepsilon = \eta/3$.
Define a probability measure $\nu$ on the closed disc $\overline{\D}$ as the
push-forward of $m$ under the boundary values of $\phi$, which we
denote by $\phi^*(\zeta)$; that is,
\[
   \nu(B) = m(\{\zeta\in\T: \phi^*(\zeta) \in B\})
\]
for all Borel sets $B \subset \overline{\D}$. Then $\nu$ is the
harmonic measure of the domain $\Omega$ with pole at $\phi(0) = 0$. Let $h$ be the
density (i.e.\ the Radon-Nikodym derivative) of $\nu_{|\T}$ (or, equivalently, of $\nu_{|K}$) with respect to $m$.
It is not difficult to see that $0 \leq h \leq 1$ a.e.

Note that  $\partial\Omega$ is rectifiable, since $\Omega\subset\D$ is a Lipschitz domain
by inspection. Therefore, according to a classical result of F.\ and M.\ Riesz
(see Theorem VI.1.2 and (1.4) on p.~202 of \cite{GarMar}), $\nu$ and the Hausdorff
1-measure $\mathcal{H}^1$ on $\partial\Omega$ are mutually absolutely continuous.
In particular, if $F\subset K$ satisfies  $\nu(F)=0$, then
$m(F) = \mathcal{H}^1(F) = 0$. This implies that $h>0$ a.e.\ on $K$. Thus, we may find
$\delta>0$ small enough so that the set $F = \{\zeta \in K: h(\zeta)>\delta\}$ satisfies
$m(K \setminus F) < \varepsilon$. Then we proceed to choose the compact set
$E \subset F$ such that $m(F\setminus E) < \varepsilon$, whence
$m(\T \setminus E) < 3\varepsilon = \eta$.

With the above preparations we may estimate, for any $f \in H^p$,
\begin{equation}\label{product}
\begin{split}
   \|(C_\phi T_g) f\|_{H^p}^p &= \int_{\T} |(T_g f) \circ \phi^*|^p\,dm
   = \int_{\overline{\D}} |T_g f|^p\,d\nu  \\
   &\geq \int_{\T} |T_g f|^p h\,dm \geq \delta  \int_E |T_g f|^p \,dm.
\end{split}
\end{equation}
Since $C_\phi T_g$ is compact $H^p \to H^p$
by Lemma~\ref{le:CphiTg}, as in the proof of that Lemma, we deduce
that $\chi_E T_g\colon H^p \to L^p(\T)$ is a compact operator.

Towards the final claim of the proposition suppose that the sequence $(f_n) \subset H^p$
is bounded and converges uniformly to $0$ on the compact subsets of $\D$.
Then also $(C_\phi T_g)f_n \to 0$ uniformly on the compact subsets of $\D$.
Therefore, by the uniqueness of the limit, every norm-convergent subsequence of
$((C_\phi T_g)f_n)$ must tend to zero. Since $C_\phi T_g$ is compact, we
actually deduce that $\|(C_\phi T_g)f_n\|_{H^p} \to 0$. Hence \eqref{product}
yields that $\|(\chi_E T_g)f_n\|_{L^p} \to 0$ as $n \to \infty$, and 
this finishes the proof of Proposition~\ref{prop:KTg}.
\end{proof}

Note that altogether the above steps complete the proof of Theorem~\ref{thm:main}.

\begin{remark}\label{comm}
(1) Theorem~\ref{thm:main0} and Corollary~\ref{cor1} state that the
linear qualitative behaviour  of non-compact Volterra operators 
$T_g\colon H^p \to H^p$ for $p \neq 2$ is very restricted  compared to that of 
arbitrary bounded operators on $H^p$. Recall that a result of 
Weis~\cite{Weis} for $L^p(0,1)$ combined with the known 
isomorphism $H^p \approx L^p(\T, m)$ for $1 < p < \infty$, see \cite{B}, imply that
if the restriction $U_{|M}$ of a given operator $U$ on $H^p$ is bounded below on some infinite-dimensional subspace 
$M \subset H^p$, then $M$ contains an isomorphic copy of 
either $\ell^p$ or $\ell^2$. 
We also remind that if $p > 2$ then any infinite-dimensional subspace $M \subset H^p$ 
contains isomorphic copies of either $\ell^p$ or $\ell^2$, but the case
$1 < p < 2$ is much more complicated,
see e.g.\ \cite[Chap.~6.4]{AK} for results of this type.
The case of the Hilbert space $H^2$ is well known and different, since
$\mathcal{S}_2(H^2)$ (i.e. the compact operators) is the unique closed ideal
in the algebra of bounded operators,
cf.\ \cite[5.1--5.2]{Pietsch}.

\smallskip

(2) The crucial estimate \eqref{product} on a compact set $E$ having large measure,
which is used in the proof of Proposition~\ref{prop:KTg},
can also be obtained in a direct fashion without recourse to the result by
F.\ and. M.\ Riesz. In fact, let $\varepsilon > 0$ be given (the exact value will
be specified later) and write $F = \{ \zeta\in K: h(\zeta)> 1/2\}$. Since
$h \leq 1$ a.e., one has
\[
   \nu(K) = \int_K h\,dm
   \leq m(\T\setminus F) \cdot \tfrac{1}{2} + m(F) 
   = \tfrac{1}{2}(1+ m(F)),
\]
and hence
\begin{equation}\label{mEstimate}
   m(F) \geq 2 \nu(K) - 1 = 1 - 2\nu(\partial\Omega\setminus K).
\end{equation}
Write $\T \setminus K = \bigcup_j I_j$, where the open arcs $I_j \subset \T$
are the (disjoint) connected components of $\T \setminus K$. For each $j$,
let $\zeta_j$ be the midpoint of $I_j$ and consider the Carleson
window $W_j = W(\zeta_j, 2\pi m(I_j))$ (if $m(I_j) \geq 1/2\pi$, take
$W_j = \D$). It follows from the geometry of $\Omega$ that
$\partial\Omega \subset K \cup \bigl(\bigcup_j \overline{W_j}\bigr)$. Since
$\nu$ is a Carleson measure, we have
$\nu(\overline{W_j}) \leq c m(I_j)$ where $c > 0$ is an absolute
constant (recall that $\phi(0) = 0$).
Consequently,
\[
   \nu(\partial\Omega\setminus K) \leq c \sum_j m(I_j) = cm(\T\setminus K)
   < c\varepsilon.
\]
In conjunction with \eqref{mEstimate}, this gives
$m(F) \geq 1 - 2c\varepsilon$. On choosing $\varepsilon = \eta/4c$ and a
compact set $E \subset F$ with $m(F\setminus E) < \eta/2$, we get
\eqref{product} with $\delta = 1/2$ and $m(\T\setminus E) < \eta/2+\eta/2 = \eta$.

\smallskip

(3) Theorem~\ref{thm:main0} also holds, with a similar proof to the case $p = 1$,
in the quasi-normed range $0 < p < 1$. However, here much less is known about the 
linear qualitative classification of operators,
so we will not include the technical details of this extension.
\end{remark}

In the proof of Lemma \ref{le:CphiTg} we utilized the compactness
of certain Volterra composition operators $C_\phi T_g$ on $H^p$
for arbitrary $g \in \BMOA \setminus \VMOA$
and an associated conformal map $\phi$. General operators of this kind were first considered
by Li and Stevi\'{c} \cite{LiSte}, and they have subsequently been studied in several papers.
As a simple by-product of part of the proof of Lemma \ref{le:CphiTg} we also record a  characterisation of the compactness of arbitrary
products $C_\phi T_g\colon H^p \to H^p$, which to the best of
our knowledge has not been made explicit in the literature.

We first point out a useful formula for the norm $\| (C_\phi T_g)f\|_{H^2}$.
Let $N_\phi$ be the Nevanlinna counting function of $\phi$, that is,
$N_\phi(w) = \sum_{z \in \phi^{-1}(w)} \log(1/ \vert z\vert)$ for
$w \in \phi(\D)$ and $w \neq \phi(0)$, and  $N_\phi(w) = 0$ for $w \notin \phi(\D)$.

\begin{lemma} \label{le:stanton}
Suppose that $\phi\colon \D\to\D$ is an analytic map, $g \in \BMOA$ and $f\in H^2$. Then
\[
   \|(C_\phi T_g)f\|_{H^2}^2 =
   |T_gf(\phi(0))|^2 + 2\int_{\D} |f(w)|^2|g'(w)|^2 N_\phi(w)\,dA(w).
\]
\end{lemma}

\begin{proof}
By applying a norm formula for composition operators due to J.\ Shapiro, see e.g.\
\cite[Thm~2.31]{CMC}, which combines the Littlewood-Paley identity with a change
of variables (a special case of Stanton's formula), it follows that
\[ \begin{split}
  \|(C_\phi T_g)f\|_{H^2}^2 &=  \vert T_gf(\phi(0))\vert^2 +
                2 \int_{\D} |(T_gf)'(w)|^2 N_\phi(w) \,dA(w) \\
     &= \vert T_gf(\phi(0)) \vert^2 + 2 \int_{\D} |f(w)|^2|g'(w)|^2 N_\phi(w)\,dA(w).
\end{split} \]
\end{proof}

\begin{corollary}\label{co:stanton}
Suppose that $\phi\colon \D \to \D$ is an  analytic map, $g \in \BMOA$ and
$0 < p < \infty$. Then $C_\phi T_g$ is compact $H^p \to H^p$
if and only if $|g'|^2 N_\phi\,dA$ is a vanishing Carleson measure.
\end{corollary}

\begin{proof}
Suppose first that $p = 2$ and let $(f_n) \subset H^2$ be any weak-null sequence, that is,
$(f_n)$ is bounded and $f_n \to 0$ uniformly on the compact subsets of $\D$. It follows
that
\[
  T_gf_n(\phi(0)) = \int_0^{\phi(0)} f_n(w)g'(w) \,dw \to 0
  \quad\text{as $n \to \infty$.}
\]
Consequently, if  $| g'|^2 N_\phi \,dA$ is a vanishing Carleson measure, then
the embedding $H^2 \to L^2(\D, | g'|^2 N_\phi\,dA)$ is compact, and
Lemma~\ref{le:stanton} implies that $\| (C_\phi T_g)f_n)\|_{H^2} \to 0$ as
$n\to \infty$.
This entails that $C_\phi T_g\colon H^2 \to H^2$ is compact.
Moreover, since $C_\phi T_g$ is bounded $H^p \to H^p$ for any $0 < p < \infty$,
the compactness of $C_\phi T_g\colon H^p \to H^p$ can be deduced  from one-sided
Krasnoselskii interpolation as in the proof of Lemma~\ref{le:CphiTg}.

Conversely, if $C_\phi T_g\colon H^p \to H^p$ is compact for some $p\in (0,\infty)$, then interpolation yields compactness for $p=2$, and again 
Lemma~\ref{le:stanton} yields that the embedding
$H^2 \to L^2(\D, |g'|^2 N_\phi\,dA)$ is compact.
This means that $|g'|^2 N_\phi\,dA$ is a vanishing Carleson measure, see
e.g.\ the proof of \cite[Thm~2.33]{CMC}.
\end{proof}



\begin{thebibliography}{99}

\bibitem{AK}
F.~Albiac and N.J.~Kalton, \emph{Topics in Banach Space Theory}, Springer, 2006.

\bibitem{Aleman06}
A.~Aleman, \emph{A class of integral operators on spaces of analytic functions},
Topics in Complex Analysis and Operator Theory, Proc.\ Winter School (Antequera, 2006),
pp.\ 3--30.

\bibitem{AleCi}
A.~Aleman and J.A.~Cima, \emph{An integral operator on $H^p$ and Hardy's inequality},
J.\ Anal.\ Math.\ 85 (2001), 157--176.

\bibitem{AleSis}
A.~Aleman and A.G.~Siskakis, \emph{An integral operator on $H^p$},
Complex Variables 28 (1995), 149--158.

\bibitem{B}
R.P.~Boas, \emph{Isomorphism between $H^p$ and $L^p$}, Amer.\ J.\ Math.\  77 (1955),
655--656.

\bibitem{CP}
F.~Cobos and L.-E.~Persson, \emph{Real interpolation of compact operators between
quasi-Banach spaces}, Math.\ Scand.\ 82 (1999), 138--160.

\bibitem{CMC}
C.C.~Cowen, and B.D.~MacCluer, \emph{Composition Operators on
Spaces of Analytic Functions}, CRC Press, 1995.

\bibitem{Duren}
P.L.~Duren, \emph{Theory of $H^p$ Spaces}, Academic Press, 1970;
reprinted by Dover, 2000.

\bibitem{Garnett}
J.B.~Garnett, \emph{Bounded Analytic Functions}, rev.\ ed., Springer, 2007.

\bibitem{GarMar}
J.B.~Garnett, and D.E.~Marshall, \emph{Harmonic Measure},
Cambridge Univ.\ Press, 2005.

\bibitem{Girela}
D.~Girela, \emph{Analytic functions of bounded mean oscillation},
Complex Function Spaces (Mekrij{\"a}rvi, 1999),
Univ.\ Joensuu Dept.\ Math.\ Rep.\ Ser.\ 4 (2001), pp.\ 61--170.

\bibitem{KX}
S.V.~Kisljakov and Q.~Xu, \emph{Interpolation of weighted and vector-valued Hardy spaces},
Trans.\ Amer.\ Math.\ Soc.\ 343 (1994), 1--34.

\bibitem{K}
M.A.~Krasnoselskii, \emph{On a theorem of M. Riesz}, Soviet\ Math.\ Dokl.\ 1 (1960), 
229--231.

\bibitem{LNST}
J.~Laitila, P.J.~Nieminen, E.~Saksman and H.-O.~Tylli,
\emph{Rigidity of composition operators on the Hardy space $H^p$},
Adv.\ Math.\ 319 (2017), 610--629.

\bibitem{LiSte}
S.~Li and S.~Stevi\'{c}, \emph{Products of Volterra type operator and composition operator
from  $H^\infty$ and Bloch spaces to Zygmund spaces}, J.\ Math.\ Anal.\ Appl.\ 345 (2008),
40--52.

\bibitem{LT77}
J.~Lindenstrauss and L.~Tzafriri, \emph{Classical Banach spaces I. Sequence spaces},
Springer, 1977.

\bibitem{Miihkinen}
S.~Miihkinen, \emph{Strict singularity of a Volterra-type integral operator
on $H^p$}, Proc.\ Amer.\ Math.\ Soc.\ 145 (2017), 165--175.

\bibitem{Pietsch}
A.~Pietsch, \emph{Operator Ideals}, North Holland, 1980.

\bibitem{Rud55}
W.~Rudin, \emph{The radial variation of analytic functions}, Duke Math.\ J.\ 22 (1955),
235--242.

\bibitem{Siskakis87}
A.G.~Siskakis, \emph{Composition semigroups and the Ces\`aro operator on $H^p$}, J.\ London
Math.\ Soc.\ 36 (1987), 153--164.

\bibitem{Siskakis90}
A.G.~Siskakis, \emph{The Ces\`aro operator is bounded on $H^1$}, Proc.\ Amer.\ Math.\ Soc.\
110 (1990), 461--462.

\bibitem{Siskakis04}
A.G.~Siskakis, \emph{Volterra operators on spaces of analytic functions---a survey},
First Advanced Course in Operator Theory and Complex Analysis (Sevilla, 2004), pp.\
51--68.

\bibitem{Weis}
L.~Weis, \emph{On perturbations of Fredholm operators in $L_p(\mu)$-spaces},
Proc.\ Amer.\ Math.\ Soc.\ 67 (1977), 287--292.

\end{thebibliography}
\end{document}